\documentclass[12pt,reqno]{amsart}

\usepackage{graphicx}
\usepackage{amsmath}
\usepackage{amssymb}
\usepackage{amscd}
\usepackage{hhline}
\usepackage{dcpic}
\usepackage{pictex}
\newtheorem{theorem}{Theorem}

\newtheorem{theoremd}{Theorem}

\newtheorem{dfn}[theoremd]{Definition}
\newtheorem{rk}[theoremd]{Remark}

\renewcommand\a{\alpha}
\renewcommand\b{\beta}

\newcommand\C{{\mathbb C}}
\newcommand\CC{{\mathcal C}}
\renewcommand\d{\delta}

\newcommand\D{{\mathcal D}}

\newcommand\E{\mathcal{E}}

\newcommand\g{\mathfrak{g}}

\renewcommand\l{\lambda}
\newcommand\La{\Lambda}
\newcommand\m{\mathfrak{m}}

\newcommand\oo{\omega}
\newcommand\op[1]{\mathop{\rm #1}\nolimits}
\newcommand\ot{\otimes}
\newcommand\p{\partial}

\newcommand\R{{\mathbb R}}
\renewcommand\t{\tau}

\newcommand\z{\sigma}
\newcommand\Z{{\mathbb Z}}
\newcommand{\comm}[1]{}

\begin{document}

 \title[Lie-B\"acklund theorem for Lie class $\oo=1$ systems]{Generalized Lie-B\"acklund theorem for \\
 Lie class $\oo=1$ overdetermined systems}
 \author{Boris Kruglikov}
 \date{}
 \address{Institute of Mathematics and Statistics, University of Troms\o, Troms\o\ 90-37, Norway. \quad
E-mail: boris.kruglikov@uit.no.}

 \vspace{-14.5pt}
 \begin{abstract}
In this paper we prove a version of Lie-B\"acklund theorem for
overdetermined systems of scalar PDEs, whose general solution
depends on 1 function of 1 variable. This generalizes the case
of involutive system of the second order on the plane treated
by E.Cartan in 1910. The approach is based on the geometric theory
of PDEs and Tanaka theory. Many examples are provided.
 \end{abstract}

 \maketitle

\section*{Introduction}

Consider the space $J^k\pi$ of $k$-jets of sections of a vector bundle $\pi:E\to M$.
This manifold carries the canonical Cartan distribution $\CC_k$ (also known as higher contact
vector distribution). Diffeomorphisms of $J^k\pi$ preserving the Cartan distributions are called
Lie transformations.

The classical Lie-B\"acklund theorem states that a Lie transformation $F:J^k\pi\to J^k\pi$
is the prolongation of a Lie transformation $f$ of $J^\epsilon\pi$: $F=f^{(k-\epsilon)}$,
where $\epsilon=1$ and $f$ is a contact transformation of $J^1\pi$ in the case $\op{rank}\pi=1$,
while $\epsilon=0$ and $f$ is a diffeomorphism of $J^0\pi$ (point transformation)
for $\op{rank}\pi>1$.

Various generalizations of this to Lie transformations of differential equations $\E\subset J^k\pi$
have been discovered since, and they became known as Lie-B\"acklund (type) theorems.
Namely an internal transformation is a symmetry of the induced distribution $\CC_\E=\CC\cap T\E$ on $\E$.
Lie-B\"acklund theorem holds for $\E$ if the internal symmetries coincide with the external symmetries,
which are (restrictions of) those Lie transformations of the ambient jet-space that preserve $\E$.

An instance of such theorem for scalar second order equations was proved in \cite{C$_2$} (parabolic equation)
and \cite{GK} (hyperbolic and elliptic PDE). For many classes of equations ($\CC$-general, normal) this
phenomenon was established in \cite{KLV,AKO}.

But Lie-B\"acklund theorem does not always hold. It is easily seen to fail for a single scalar PDE
of 1st order $F(x^1,\dots,x^n,u,u_1,\dots,u_n)=0$ for which the internal symmetry group is
infinite-dimensional. Another important counter-example
is the Hilbert-Cartan equation $y'=(z'')^2$, for which the external symmetry group
is 9-dimensional\footnote{This is the parabolic subgroup $P_2\subset G_2$
giving the contact gradation.}
vs. 14-dimensional internal symmetry group, though the latter coincides
with the space of generalized Lie-B\"acklund symmetries \cite{AKO,Ker}.

More general Monge equations $y^{(m)}=F(x,y,\dots,y^{(m-1)},z,\dots,z^{(n)})$ were considered in \cite{AK}.
If $m=1,n>2$ or $m>1$ then Lie-B\"acklund theorem holds provided the equation is non-degenerate
(this condition meaning that $F$ is not affine in $z^{(n)}$ cannot be dropped) and the
external symmetries concern the ambient mixed jet-space $J^{m,n}(\R,\R\times\R)$.

The Monge equations are included into a more general class of Monge systems (systems of ODEs with the
minimal degree under-determinacy), which are the natural target of reduction for so-called Lie
class $\oo=1$ compatible overdetermined PDE systems. These are the systems $\E$ whose general
solution depends on 1 function of 1 variable.

Lie class $\oo=1$ PDE systems are integrable by ODE methods \cite{L,K$_3$}, and they
often arise as symmetry reductions of more complicated PDEs
(a Darboux integrable or semi-integrable equation coupled with
one intermediate integral is a particular case of class $\oo=1$ system).
Lie-B\"acklund theorem clearly fails for Lie class $\oo=1$ overdetermined PDE systems, since
the internal symmetry group is always infinite-dimensional while the external group is usually not.

An important example of such $\E$ constitute involutive overdetermined systems of 2nd order scalar
PDEs on the plane considered by E.Cartan ("le th\'eor\`eme imporant" \cite[\S26]{C$_2$}).
He indicated a generalization of Lie-B\"acklund theorem to this case,
namely the internal group is changed to the symmetry group
of the reduction $(M,\Delta)$ of $(\E,\CC_\E)$ by the Cauchy characteristic $\Pi\subset\CC_\E$
and then it is bijective with the external symmetry group.

This was generalized to involutive 2nd order PDE systems in $n$-dimensions in \cite{Y$_1$,Y$_2$}.
The purpose of this paper is to prove

 \begin{theorem}\label{THM}
Let $\E$ be a Lie class $\oo=1$ overdetermined compatible system of PDEs
of orders greater than 1. If $\E$ is sufficiently non-degenerate, then the symmetry algebra
of the reduction $(M,\Delta)$ is equal to the Lie algebra of external symmetries of $\E$.
 \end{theorem}

The requirement of sufficient non-degeneracy is technical and will be formulated,
after we introduce some preliminary material, in Section \ref{S4}.
Then in Section \ref{S5} we will prove the main result.

We will apply it to calculate the group of contact transformations of some model overdetermined
systems of PDEs. To find the complete symmetry group we will elaborate upon the Tanaka theory
of symmetries of vector distributions.

We shall also discuss limitations of the theorem, by showing examples of degenerate systems
when the external and internal symmetries are different.
Here is the version of Theorem \ref{THM} for two different equations
(by micro-local we mean in a neighborhood of a point on the equation).

 \begin{theorem}\label{THM2}
Two Lie class $\oo=1$ overdetermined compatible sufficiently non-degenerate systems
$\E,\E'$ of orders greater than 1 are micro-locally equivalent if and only if their
reductions $(M,\Delta)$, $(M',\Delta')$ are.
 \end{theorem}

\medskip

\textsc{Acknowledgment.} I thank Kazuhiro Shibuya and Toshihiro Shoda for invitation to
the conference "Differential Geometry and Tanaka Theory, Differential System and Hypersurface Theory"
in 2011, where I had useful discussions with Keizo Yamaguchi and Ian Anderson.
\textsc{Maple} package \textit{DifferentialGeometry} was used in several calculations.

\section{Symmetries of differential equations}\label{S1}

A differential equation $\E$ is often treated geometrically as a submanifold in the space
of jets $J^k(B,F)$, where $B$ is the space of independent variables $x=(x^1,\dots,x^n)$ and
$F$ is the space of dependent variables $u=(u^1,\dots,u^m)$, which we write as $u=u(x)$.
A choice of these coordinates on $J^0=B\times F$ yields canonical coordinates on the space $J^k$:
$(u^j_\z)$, where multiindex $\z=(i_1,\dots,i_n)$ has length $|\z|=\sum_{a=1}^n i_a\le k$.

We will assume regularity, i.e. that the projection maps $\pi_{k,l}:J^k\to J^l$
have constant ranks when restricted to $\E$ and its prolongations. Usually it is assumed that
the equation has pure order $k$ and then $\pi_{k,k-1}:\E\to J^{k-1}$ is surjective. But we will
allow PDEs in the system $\E$ to have different orders (of which $k$ is the maximal).
For the setup of this theory we refer to \cite{KL}.

The Cartan distribution $\CC_k$ on $J^k$ is given as the kernel of the forms
$\theta_\z^j=du^j_\z-\sum_{i=1}^nu^j_{\z+1_i}dx^i$ for all $j\in\{1,\dots,m\}$ and $|\z|<k$.
It is generated as follows:
 $$
\CC_k=\Bigl\langle\D_{x^i}=\p_{x^i}+\sum_{j;|\t|<k}u^j_{\t+1_i}\p_{u^j_\t},\ \p_{u^j_\z}\ :
\ 1\le i\le n,|\z|=k\Bigr\rangle
 $$
Cartan distribution $\CC_l$ on $J^l$ can be lifted to $J^k$ via $\pi_{k,l}^{-1}$.
In particular $\CC_1$ is the classical contact distribution.

These distributions for different $l$ are related by the formula\footnote{For brevity sake we will
write $[\nabla,\Delta]$ instead of more appropriate "distribution whose module of sections equals $[\Gamma(\nabla),\Gamma(\Delta)]$" (regularity assumed). Notice that we always have
$[\nabla,\Delta]\supset\nabla+\Delta$.}
 $$
[\CC_l,\CC_l]=\CC_{l-1}.
 $$

The classical Lie-B\"acklund theorem claims that a symmetry of distribution $\CC_k$ on $J^k$ is
necessarily the lift of a contact transformation from $J^1$ if $m=1$ or a point transformation
of $J^0$ if $m>1$. The proof follows almost immediately from the above display formula.

A symmetry of $(J^k,\CC_k)$, preserving $\E$, is called an external symmetry of the system.
The Lie algebra of all external symmetries is denoted by $\op{Sym}(\E)$. The internal symmetries
$\op{Sym}_{int}(\E)$ are by definition the symmetries of the induced distribution $\CC_\E=\CC_k\cap T\E$ on $\E$.

Restriction obviously gives the homomorphism $\op{Sym}(\E)\to \op{Sym}_{int}(\E)$.
Generalized Lie-B\"acklund theorems address the occasions when it is an isomorphism.

In what follows we will consider an overdetermined system of PDEs $\E$.
It will be assumed formally integrable, meaning that for all prolongations $\E_{k+t}=\E^{(t)}$
(given as the locus in $J^{k+t}$ of defining relations for $\E$ differentiated $\le t$ times)
the projections $\pi_{k+t,k+t-1}:\E_{k+t}\to\E_{k+t-1}$ are affine bundles.
In other words, all the compatibility conditions vanish.

For technical reasons we will need to prolong to the level $s=k+t$, when $\E_s$ is involutive.
This means vanishing of the Spencer cohomology $H^{i,j}(\E)=0$, $\forall i\ge s,j\ge0$, or
equivalently fulfilment of the Cartan test \cite{C$_3$,KLV}.
Then for $i\ge s$ the symbol
 $$
g_i=\op{Ker}(d\pi_{i,i-1}:T\E_i\to T\E_{i-1})
 $$
has dimension growth in accordance with Hilbert polynomial \cite{KL}.
If $x_i\in\E_i$ and $\pi_{i,0}(x_i)=(x,y)$, we can identify $g_i(x_i)\subset S^iT^*_xB\ot T_yF$.

\section{Lie class $\oo=1$ compatible systems}\label{S2}

The (complex) {\it characteristic variety\/} can be defined as projectivization
of the set of complex characteristic covectors
 $$
\op{Char}^\C_{x_i}(\E)=\mathbb{P}\{p\in T_x^*B\ot\C\ :\ p^k\ot T_yF\cap g_k^\C\ne0\}.
 $$

If $\op{Char}^\C(\E)=\emptyset$ then the system is of finite type, i.e. its solution space if finite-dimensional.
The next by complication case is when $\op{Char}^\C(\E)$ consists of one point
counted with multiplicity (then the characteristic variety must be real).
These are the systems of Lie class $\oo=1$ and their solutions are parametrized 1 function of
1 argument\footnote{In terminology of Elie Cartan $\oo$
is the {\it Cartan integer\/} $s_1$ (provided the Cartan character is 1: $s_2=0$).}.

Another way to describe such systems is the following. For large $i$ (precisely starting
from the level when $\E_i$ becomes involutive) the symbol of the system stabilizes: $\dim g_i=1$.
We refer to \cite{K$_2$} for a detailed discussion of these systems.

By a theorem of S. Lie compatible PDE systems of class $\oo=1$ are integrable via ODEs \cite{L}.
The proof given in \cite{K$_3$} uses the following observation: For such systems $\E$,
starting from the involutivity jet-level, the Cartan distribution $\CC_\E$ has
rank equal to $(n+1)$ and it contains the subspace $\Pi$ of Cauchy characteristics of rank $(n-1)$.
Since the prolongation, required to achieve involutivity, changes neither the class $\oo$ nor
the external symmetry group, we will assume that already $\E$ is involutive
(this does not restrict generality of the results).

The (local) quotient by the foliation tangent to $\Pi$ maps $(\E,\CC_\E)$ to a manifold
with rank 2 distribution. Let us call this pair $(M,\Delta)$ the reduction of our system.

Shifts along Cauchy characteristic direction are {\it trivial symmetries\/} of $\E$, so the space
of internal symmetries is always infinite (while this for external is usually not).
To compensate this we consider the symmetry algebra of the reduction $\op{Sym}(M,\Delta)$
(all our considerations are local). Since the subalgebra of trivial symmetries is an ideal
in $\op{Sym}_{int}(\E)$, we have the homomorphism induced by the restriction
 \begin{equation}\label{LBT}
\op{Sym}(\E)\to\op{Sym}(M,\Delta).
 \end{equation}
The goal of the paper is to demonstrate that this is an isomorphism (under certain
"general position" assumptions).

\smallskip

For simplicity of the exposition we restrict in the next section to the case of scalar PDEs ($m=1$)
on the plane ($n=2$). We will denote a system of differential equations of orders $k_1,\dots,k_r$
by $\sum_1^rE_{k_i}$.

A zoo of such systems is given in \cite{K$_2$}.
In this reference and \cite{K$_3$} it is shown that internal geometry of linearizable systems is quite simple:
their reduction correspond to Goursat distributions.

Recall that for a distribution $\Delta$ the {\it weak derived flag\/} is defined by
$\{\Delta_1=\Delta,\Delta_{i+1}=[\Delta,\Delta_i]\}$, and the {\it strong derived flag\/} is given
by the formula $\{\nabla_1=\Delta,\nabla_{i+1}=[\nabla_i,\nabla_i]\}$.
The sequences $\{\dim(\Delta_{i+1}/\Delta_i)\}$ and $\{\dim(\nabla_{i+1}/\nabla_i)\}$
will be called weak resp.\ strong {\it growth vectors\/}.

Goursat distributions have both weak and strong growth vectors $(2,1,1,\dots,1)$ and are
isomorphic to the Cartan distribution on the jet space $J^k(\R,\R)$. Thus the internal geometry
of the linear class $\oo=1$ overdetermined compatible systems is trivial.

On the other hand the external geometry (which is governed by the pseudogroup of
point triangular transformations) is rich and is characterized by differential invariants.

An interesting example of systems of type $E_2+E_3$ can be found already in Cartan (\cite{C$_2$}, p.147).
Modifying his PDEs a bit we get:
 $$
\E:\{u_{xy}=0,u_{yyy}=0\},\qquad \bar\E:\{u_{yy}=0,u_{xxy}=0\}.
 $$
Internally $\E\simeq\bar\E$ since the growth vectors of both reductions are $(2,1,1,1)$,
but the systems are not equivalent externally because the second order equations are hyperbolic and
parabolic respectively. Thus generalized Lie-B\"acklund theorem fails in this case.
By \cite{C$_2$} such situation is not possible for
systems of type $2E_2$.

\section{Two examples of class $\oo=1$}\label{S3}

Let us study some partial cases with $n=2$ independent variables.
We will use the classical notations $p=u_x$, $q=u_y$, $r=u_{xx}$, $s=u_{xy}$,
$t=u_{yy}$, $\a=u_{xxx}$, $\b=u_{xxy}$, $\gamma=u_{xyy}$, $\d=u_{yyy}$.

\smallskip

Consider at first the system $E_2+E_3$. Without loss of generality we can assume
that the characteristic\footnote{This vector belongs to the annihilator of
the characteristic covector $dx+A\,dy$.} is $\p_y-A\,\p_x$, where $A$ is a function on 2-jets
(more precisely on equation $E_2$). Thus the system $\E$ writes
 $$
t=F(r,s,\dots),\quad \b=A(r,s,\dots)\a+B(r,s,\dots),
 $$
where dots mean terms of order 1 (and $A$ satisfies $A^2=F_sA+F_r$).
The Cauchy characteristic vector field for the distribution $\CC_\E$ is
$\xi=\D_y-A\,\D_x+\varrho\,\p_\a$ for some function $\varrho$, and
it is transversal to $\Sigma:y=\op{const}$.

Consequently the reduced rank 2 distribution (on the quotient $M$) can be interpreted as the following
distribution on $\Sigma$: $\Delta^2=\CC_\E\cap T\Sigma$ (the value of constant plays no role).
In the canonical coordinates
 $$
\Delta=\langle\D_x=\p_x+p\,\p_u+r\,\p_p+s\,\p_q+\a\p_r+(A\a+B)\,\p_s, \p_\a\rangle.
 $$
This distribution has growth vector $(2,3,4,\dots)$ and so is de-prolongable, i.e. locally
$\Delta=\mathbb{P}(\bar\Delta)$ for some rank 2 distribution $\bar\Delta$ on 6-dimensional
manifold $\bar M$. Indeed, $\p_\a$ is the Cauchy characteristic for the derived distribution
$\Delta_2$.

Thus the weak derived flag $\{\Delta_i\}$ differs drastically from the strong derived flag $\{\nabla_i\}$.
One possibility is to perform de-prolongation (in this case it is easy - to restrict $\Delta_2$ to
the transversal $\bar M:\a=\op{const}$), but we will instead consider the strong derived flag.
It has the following generators, provided the strong growth vector is $(2,1,1,2,1)$:\footnote{This
is the case of general position; elsewise there are more de-prolongations
or there exists a first integral of the distribution.}
 \begin{alignat*}{2}
& \nabla_1=\langle e_1,\p_\a\rangle,         && e_1=-\D_x\\
& \nabla_2=\nabla_1+\langle e_1'\rangle,     && e_1'=[e_1,\p_\a]=\p_r+A\,\p_s\\
& \nabla_3=\nabla_2+\langle e_2\rangle,      && e_2=[e_1,e_1']=\p_p+A\,\p_q+(e_1'(B)-\D_x(A))\p_s\\
& \nabla_4=\nabla_3+\langle e_3,e_3'\rangle, && e_3=[e_1,e_2]=\p_u+\dots,\\
 & \hspace{120pt} && e_3'=[e_1',e_2]=e_1'(A)\p_q+C\p_s,\\
& \nabla_5=\nabla_4+\langle e_4\rangle=TM,   && \hspace{20pt} e_4=[e_1,e_3],\ [e_1,e_3']=[e_1',e_3]
 \text{ or } [e_1',e_3'],
 \end{alignat*}
where the coefficients, like $C$ or dots, will not be indicated explicitly (the growth vector of
the reduction on $\bar M$ is $(2,1,2,1)$ with the flag generated by vectors
$\langle e_1,e_1';e_2;e_3,e_3';e_4\rangle$).

If we know the vertical distribution\footnote{Notice that it is in the kernel of the bracket-map
$\La^2\nabla_2\to\nabla_3$ since $\p_\a$ is the Cauchy characteristic for $\nabla_2$.}
$\nabla_2'=\langle e_1',\p_\a\rangle$,
we can construct the distribution
$\nabla_4'=[\nabla_2',\nabla_3]=\langle e_1,e_1',e_2,e_3',\p_\a\rangle$.

Distribution $\Delta$ will be called {\it sufficiently non-degenerate\/} if $\nabla_5'=[\nabla_2',\nabla_4']$
has rank $6$ (i.e. we can use $[e_1',e_3']$ for $e_4$). In this case it contains the
fiber $\langle \p_p,\p_q,\p_r,\p_s,\p_\a\rangle$ of the projection
$d\pi_{3,0}:T\E\to TJ^0$ as a codimension 1 subspace,
so we re-cover the contact distribution on $J^1$ (lifted to $\E$):
$\CC_1=\nabla_5'+\langle\xi\rangle$.

\smallskip

The second particular type we want to examine is the system $3E_3$, which writes
 $$
\b=F(\a,\dots),\quad \gamma=G(\a,\dots),\quad \d=H(\a,\dots),
 $$
now dots mean terms of order 2. The condition of 1 common characteristic implies
$G_\a=F_\a^2$, $H_\a=F_\a^3$ and we assume $\E$ fully nonlinear $F_{\a\a}\ne0$ (this
is {\it sufficient non-degeneracy\/}; the quasi-linear case yields formulae similar
to the first example, so most of them also work).

The characteristic is then $\p_y-F_\a\p_x$, whence the Cauchy characteristic of $\CC_\E$
is $\xi=\D_y-F_\a\D_x+\varrho\,\p_\a$ for some function $\varrho$ on $\E$ and we again choose
the transversal as $\Sigma:y=\op{const}$.

This will be identified with the quotient $M$ and the induced rank 2 distribution is
 $$
\Delta=\langle\D_x=\p_x+p\,\p_u+r\,\p_p+s\,\p_q+\a\p_r+F\p_s+G\p_t, \p_\a\rangle.
 $$
Now due to full nonlinearity the square of this distribution does not have Cauchy characteristics
(i.e. it is not de-prolongable). This implies that the
first 3 elements of the weak and strong derived flags are the same, and so in our
case\footnote{This equality of weak and strong derived flags
can fail already when the length of the weak growth vector is 5, even in fully nonlinear case.}
both growth vectors are $(2,1,2,3)$.

The flags are given by
 \begin{alignat*}{2}
& \nabla_1=\langle e_1,e_1'\rangle,           && e_1=-\D_x,e_1'=\p_\a\\
& \nabla_2=\nabla_1+\langle e_2\rangle,       && e_2=[e_1,e_1']=\p_r+F_\a\p_s+F_\a^2\p_t\\
& \nabla_3=\nabla_2+\langle e_3,e_3'\rangle,  && e_3=[e_1,e_2]=\p_p+F_\a\p_q+C\p_s+D\p_t,\\
& \hspace{140pt} && e_3'=[e_1',e_2]/F_{\a\a}=\p_s+2F_\a\p_t,\\
& \nabla_4=\nabla_3+\langle e_4,e_4',e_4''\rangle=TM, && \hspace{20pt} e_4=[e_1,e_3]=\p_u+\dots,\\
& \hspace{20pt} e_4'=[e_1,e_3']=\p_q+\tilde C\,\p_s+\tilde D\,\p_t,\hspace{-20pt} && \hspace{40pt} e_4''=[e_1',e_3']/(2F_{\a\a})=\p_t.
 \end{alignat*}
Thus if we know the vertical distribution $\langle e_1'\rangle\subset\Delta$ we can re-cover
the Cartan distribution on $J^1$: $\CC_1=[e_1',\nabla_3]+\langle\xi\rangle$ (here we use the fact that
$e_4'=[e_1',e_3]/F_{\a\a}\,\op{mod}\nabla_3$).

Another way is to re-cover the contact distribution on $J^2$:
$\CC_2=[e_1',[e_1',\nabla_2]]+\langle\xi\rangle$
and then get $\CC_1=[\CC_2,\CC_2]$.

\smallskip

Of course, in both cases we do not know the vertical distributions from internal viewpoint.
However generically (if the system is sufficiently non-linear) they rotate inside the corresponding
internally-canonical distributions when moving along the Cauchy characteristic.
This crucial observation will allow us to prove the equivalence.

\section{Sufficient non-degeneracy}\label{S4}

We need to specify what conditions on $\E$ should be imposed to obtain the generalized
Lie-B\"acklund theorem. The first one is:
 $$
\text{Nonlinearity of }\E\ \Leftrightarrow\ \Delta\text{ is not Goursat.}
\leqno(N)
 $$
Thus the strong growth vector of $\Delta$ is $(2,1,\dots,1,2,\dots)$, i.e. there exists $s>2$ such that
$\dim(\nabla_s/\nabla_{s-1})=2$ and this is the first $2$ after a sequence of $1$ ($s$ is minimal).

Then $\dim(\nabla_s)=s+2$ and this distribution has $(s-3)$-dimensional sub-distribution
$\square_{s-3}\subset\nabla_{s-3}$ of Cauchy characteristics for $\nabla_s$
(that are also Cauchy characteristics for $\nabla_{s-2}$).

Denote the projection along Cauchy characteristic space $\Pi\subset\CC_\E$ by $\varpi:\E\to M$; it maps
$\CC_\E\to\Delta$. The strong growth vector of $\CC_\E$ is $(n+1,1,\dots,1,2,\dots)$, where the
first dimension $2$ occurs at the place $s$. The strong derived flag is
$\hat\nabla_i=\varpi^{-1}(\nabla_i)$.
Let $\hat\square_{s-3}=\varpi^{-1}(\square_{s-3})\subset\hat\nabla_{s-3}$
be the space of Cauchy characteristics for $\hat\nabla_s$ (and $\hat\nabla_{s-2}$).

Let $\upsilon_{s-2}=\hat\nabla_{s-2}\cap\op{Ker}(d\pi)$, where $\pi:\E\to B$ is the
natural projection in jets; $\op{rank}(\upsilon_{s-2})=s-2$. Then
 \begin{equation}\label{dag}
\hat\square_{s-3}\subset\upsilon_{s-2}+\Pi\subset\upsilon_{s-2}+\CC_\E=\hat\nabla_{s-2},
 \end{equation}
and both inclusions have codimension $1$: $\op{rank}(\hat\square_{s-3})=n+s-4$,
$\op{rank}(\upsilon_{s-2}+\Pi)=n+s-3$, $\op{rank}(\hat\nabla_{s-2})=n+s-2$.

Since $\Pi$ is the space of Cauchy characteristics,
we have $[\Pi,\hat\nabla_{s-2}]=\hat\nabla_{s-2}$.
Our second requirement is that the space $\upsilon_{s-2}$ rotates along $\Pi$:
 $$
[\Pi,\upsilon_{s-2}]=\hat\nabla_{s-2}.
\leqno(R)
 $$
This condition holds if the system exhibits some degree of non-linearity;
for instance, a system of type $kE_k$ with $n=2$ satisfies (\textsl{R\,})
whenever the equations in it are not all quasi-linear.


Denote the operation of taking bracket with $\upsilon_{s-2}$ by $\op{ad}_{\upsilon_{s-2}}$.
We have $\op{ad}_{\upsilon_{s-2}}(\hat\nabla_{s-1})\subset\hat\nabla_s$.
Let us continue taking brackets and denote the limit by $\op{ad}^\infty_{\upsilon_{s-2}}(\hat\nabla_{s-1})$.
Our third requirement is that this latter generates the Cartan distribution:
 $$
(k-s+1)\text{-strong derived of }\op{ad}^\infty_{\upsilon_{s-2}}(\hat\nabla_{s-1})
\text{ is equal to }\CC_1.
\leqno(G)
 $$
Actually, if $(k-s+2)$ is less than the minimal order of $\E$, we can formulate the condition
in a simpler form:
 $$
\op{ad}^\infty_{\upsilon_{s-2}}(\hat\nabla_{s-1})=\CC_{k-s+2}.
\leqno(G')
 $$
This implies, in particular, that the Cartan distribution $\CC_\E$ on $\E$
is  completely non-holonomic, i.e. its iterated brackets generate $T\E$.

Finally let us consider the condition on the space of Cauchy characteristics,
which arise only in the case $n>2$ (when $\dim\Pi=n-1>1$).
Since the inclusion $\Pi+\upsilon_{s-2}\subset[\Pi,\hat\nabla_{s-2}]=\hat\nabla_{s-2}$
has codimension 1, condition (\textsl{R\,})
implies existence of codimension 1 sub-distribution $\Pi_1\subset\Pi$ such
that $[\Pi_1,\upsilon_{s-2}]\subset\Pi+\upsilon_{s-2}$. We assume, in addition, to
(\textsl{R\,}), that we have equality instead of the general inclusion.

Then there exists a codimension 1 sub-distribution $\Pi_2\subset\Pi_1$ such
that $[\Pi_2,\upsilon_{s-2}]\subset\Pi_1+\upsilon_{s-2}$. Again, we strengthen
this to be equality and continue. To summarize we arrive to the filtration
 $$
\Pi=\Pi_0\supset\Pi_1\supset\Pi_2\dots\supset\Pi_{n-2},
 $$
where $\op{rank}(\Pi_i)=n-i-1$ and $[\Pi_i,\upsilon_{s-2}]\subset\Pi_{i-1}+\upsilon_{s-2}$.
Our last condition, strengthening (\textsl{R\,}), is that the filtration rotates along $\upsilon_{s-2}$,
i.e.\ the last defining relation is the equality:
 $$
[\Pi_i,\upsilon_{s-2}]=\Pi_{i-1}+\upsilon_{s-2},
\leqno(R_+)
 $$
where we let $\Pi_{-1}=\CC_\E=\varpi^{-1}(\Delta)$ to include the condition (\textsl{R\,}).
Notice that there are no difference between (\textsl{R\,}) and (\textsl{R$_+$}) for $n=2$.

 \begin{dfn}
An overdetermined compatible system $\E$ of class $\oo=1$ is sufficiently non-degenerate
if it satisfies conditions (\textsl{N}), (\textsl{R$_+$}) and (\textsl{G}).
 \end{dfn}

Now we can prove two theorems from the Introduction.

\section{Proof of the main results}\label{S5}

Let us start by considering the case, when the base $B$ has dimension $n=2$,
so that the space $\Pi$ of Cauchy characteristics is 1-dimensional.

It is obvious that an external symmetry of $\E$ induces an internal
symmetry of the reduction $(M,\Delta)$. We have to show that the inverse exists and is unique.

Since $\Pi$ is the space of Cauchy characteristics, we have $(\E,\CC_\E)\simeq(M,\Delta)\times(\Pi,\Pi)$.
This local diffeomorphism does not show the contact structure.
The latter can be uncovered due to condition (\textsl{R\,}) as follows.

By (\ref{dag}) we can locally identify the space $(\E,\upsilon_{s-2}+\Pi)$ with
$P_{s-2}=\{V^{s-2}:\square_{s-3}\subset V^{s-2}\subset\nabla_{s-2}\}$.
Since every $(s-2)$-dimensional space $V^{s-2}\in P_{s-2}$ is squeezed between the subspaces
of $TM$ of dimensions $(s-3)$
and $(s-1)$, this $P_{s-2}$ is fibered over $M$ with fibers of dimension 1, and the
condition (\textsl{R\,}) means that the line field $V^{s-2}/\square_{s-3}$ rotates along the fiber,
which can be locally identified with $\Pi$.

Indeed, the above construction corresponds to the prolongation of rank 2 distributions
as follows. Recall that for a rank 2 distribution $\bar\Delta$ on $\bar M$ its prolongation
$\hat M$ is the manifold of all 1-dimensional subspaces
$\ell\subset\bar\Delta$, with the natural projection $\rho:\hat M\to\bar M$.
The fiber of $\rho$ over $\bar x\in\bar M$ is $\mathbb{P}\Delta_{\bar x}\simeq S^1$ and
the natural lift of the distribution is given by the formula
$\hat\Delta_\ell=d_\ell\rho^{-1}(\ell)$. This is a rank 2 distribution on $\hat M$ with
the derived rank 3 distribution equal to $[\hat\Delta,\hat\Delta]=d\rho^{-1}(\bar\Delta)$.
The space of Cauchy characteristics of the latter is the fiber $TS^1$.

Now letting $(\bar M,\bar\Delta)=(M,\nabla_{s-2})/\square_{s-3}=(\E,\hat\nabla_{s-2})/\hat\square_{s-3}$
(this is a local construction, quotient by $\square_{s-3}$ is possible as it is the space of
Cauchy characteristics), we can locally identify the distribution $\upsilon_{s-2}+\Pi$ on $\E$
with the distribution $\hat\Delta\times\square_{s-3}$ on $\hat M\times\R^{s-3}$.

This construction allows us to uniquely (locally) recover the pair $(\hat\nabla_{s-2},\upsilon_{s-2})$
on $\E$ from the internal geometry of the distribution $\Delta$ on $M$. Since, by condition (\textsl{G\,}),
the pair $(\hat\nabla_{s-2},\upsilon_{s-2})$ generates the contact distribution on the space of jets,
the symmetries of $(M,\Delta)$ are bijective with external Lie infinitesimal transformations of $\E$.

Consider now the case $n>2$. By definition the bracket with $\Pi_1$ acts trivially on the
distribution $\hat\nabla_{s-2}/\hat\square_{s-3}$.
The previous construction identifies the line bundle $\Pi/\Pi_1$ with the
projectivization of the rank 2 distribution $\bar\Delta=\hat\nabla_{s-2}/\hat\square_{s-3}$.
Thus the distribution $\bar\Delta+\Pi/\Pi_1$ can be identified with the
square $[\hat\Delta,\hat\Delta]$ of the prolonged distribution $\hat\Delta$.
This latter is identified with $\upsilon_{s-2}/\square_{s-3}+\Pi/\Pi_1$, with
$\Pi/\Pi_1$ corresponding to its vertical line subbundle.

By condition (\textsl{R$_+$}) this line bundle rotates along $\Pi_1/\Pi_2$ and does not
rotate along $\Pi_2$. Thus we can identify the line bundle $\Pi_1/\Pi_2$ with the prolongation
of $\hat\Delta$. Continuing in the same way we identify $\Pi$ with $(n-1)$-st iterated prolongation
of the rank 2 distribution $\bar\Delta$ over $M$, and so we uniquely recover the contact
geometry of $\E$ from the internal geometry of $(M,\Delta)$.

This proves Theorem \ref{THM}; Theorem \ref{THM2} follows shortly.

\section{Application I}\label{S6}

Consider the following overdetermined system $\E_k\subset J^k(\R^2)$:
 $$
\E_k:\quad\Bigl\{u_{k-i,i}=\frac{\lambda^{i+1}}{i+1}:0\le i\le k\Bigr\}
 $$
(we use jet multi-index notations $u_{10}=u_x$, $u_{01}=u_y$, $u_{20}=u_{xx}$ etc;
parameter $\lambda$ is to be excluded).

For $k=1$ this is the equation $2u_y=u_x^2$, which is equivalent via a potential change of variables
to the equation of gas dynamics $u_y=u\,u_x$. Both have infinite-dimensional contact symmetry algebra
$\mathfrak{cont}(\R^3)$.

For $k=2$ this is the celebrated involutive second order PDE system considered by E.Cartain in 1893 and
1910 \cite{C$_1$, C$_2$}:
 \begin{equation}\label{Ceq}
\Bigl\{u_{xx}=\l,\ u_{xy}=\frac{\l^2}2,\ u_{yy}=\frac{\l^3}3\Bigr\}.
 \end{equation}
He has shown that the contact symmetry group is
the (split) exceptional Lie group $G_2$.

 \begin{theorem}\label{gCE}
The algebra of contact symmetries of the equation $\E_k$ for $k>2$ has dimension $k(k+1)/2+6$ and
is isomorphic to the semi-direct product $\mathfrak{n}_{k+1}\rtimes\op{gl}_2$, where
$\mathfrak{n}_k$ is a nilpotent Lie algebra of length $k$.
 \end{theorem}

Every contact vector field on $J^1=J^1(\R^2)$ has the form (here $\D_x^1=\p_x+u_x\p_u$,
$\D_y^1=\p_y+u_y\p_u$):
 $$
X_f=-f_{u_x}\D_x^1-f_{u_y}\D_y^1+f\,\p_u+\D_x^1(f)\,\p_{u_x}+\D_y^1(f)\,\p_{u_y},
 $$
where $f\in C^\infty(J^1)$ is
the generating function (contact Hamiltonian). Let us indicate the generating functions $f$ of
a basis of $\op{Sym}(\E_k)$:
 \begin{multline*}
x^iy^j\ (i+j<k),\ \
u_x,\ u_y,\ u+y\cdot u_y,\ (k+1)u-x\,u_x,\\
y\cdot u_x+x^k/k!,\ (k-1)y\cdot u-xy\cdot u_x-y^2\cdot u_y-x^{k+1}/(k+1)!\,.
 \end{multline*}
Since the prolongation of the field $X_f$ to the space $J^k$ is given by the formula
($\z$ is the multi-index of the derivation, $\D_\z$ is the iterated total derivative and
$\D^k_i=\sum\limits_{|\z|\le k}u_{\z+1_i}\p_{u_\z}$ is the truncated total derivative \cite{KLV})
 $$
\hat X_f=-f_{u_x}\D_x^k-f_{u_y}\D_y^k+\sum_{|\sigma|\le k}\D_\z(f)\p_{u_\z},
 $$
we see that the first elements of the above collection act trivially on $u_{k-i,i}$. The same concerns
the translations $u_x,u_y$. The next elements in the first line are scalings and it is trivial to check
they preserve the system $\E_k$. The last elements (second line) of the above collection
have the following prolongations:
 \begin{multline*}
-y\,\p_x+\tfrac{x^k}{k!}\,\p_u+\tfrac{x^{k-1}}{(k-1)!}\,\p_{u_x}+u_x\,\p_{u_y}
+\tfrac{x^{k-2}}{(k-2)!}\,\p_{u_{xx}}+u_{xx}\,\p_{u_{xy}}+2u_{xy}\,\p_{u_{yy}}+\\
+\p_{u_{k,0}}+u_{k,0}\,\p_{u_{k-1,1}}+2u_{k-1,1}\,\p_{u_{k-2,2}}+\dots+ku_{1,k-1}\,\p_{u_{0,k}};
 \end{multline*}
 \begin{multline*}
xy\,\p_x+y^2\,\p_y+((k-1)yu-\tfrac{x^{k+1}}{(k+1)!})\,\p_u+\\
+((k-2)y u_x-\tfrac{x^k}{k!})\,\p_{u_x}+((k-1)u-xu_x+(k-3)yu_y)\,\p_{u_y}+\dots-\\
-(x+y u_{k,0})\,\p_{u_{k,0}}-\sum_{1\le i\le k}(i\,x\, u_{k-i+1,i-1}+(i+1)y\, u_{k-i,i})\,\p_{u_{k-i,i}}.
 \end{multline*}
Now it is straightforward to check these preserve the system $\E_k$.

Let us consider the reduction of $\E_k$ by the Cauchy characteristic, which is
the vector field $\D_y^k-\l\,\D_x^k+\varrho\,\p_\lambda$, where $\l=u_{k,0}$ and
the precise value of $\varrho$ is not important.

The quotient $(M,\Delta)$ is (locally) isomorphic to the intersection with the transversal
$\{y=\op{const}\}$, so that $\Delta=\langle\D_x^k,\p_\lambda\rangle$, where
 \begin{multline*}
\D_x^k=\p_x+u_{10}\,\p_u+u_{20}\,\p_{u_{10}}+u_{11}\,\p_{u_{01}}+\dots+\\
+\lambda\,\p_{u_{k-1,0}}+\tfrac12\lambda^2\,\p_{u_{k-2,1}}+\dots+\tfrac1k\lambda^k\,\p_{u_{0,k-1}}.
 \end{multline*}
Denoting $w_i^j=u_{ij}$, $0\le i+j\le k-1$, we see that the above rank 2 distribution
corresponds to the following underdetermined ODE system on $w=(w^0,\dots,w^{k-1})$ as a function of $x$
(we understand $w_i^j=\p_x^i(w^j)$):
 $$
\mathcal{Y}_k:\quad \bigl\{w_k^0=\lambda,\ w^1_{k-1}=\tfrac12\lambda^2,\ \dots,\ w_1^{k-1}=\tfrac1k\lambda^k
\bigr\}.
 $$
By Theorem \ref{THM} the algebra of external symmetries $\op{Sym}_{\mathfrak{cont}}(\E_k)$
is isomorphic to the algebra of internal symmetries of the above Monge equation
$\mathcal{Y}_k\subset J^{k,k-1,\dots,2,1}(\R,\R^k)$ (considered as a submanifold in the mixed-jet
space)\footnote{Notice that $\mathcal{Y}_2\subset J^{2,1}(\R,\R\times\R)$ is the Hilbert-Cartan
equation. In \cite{AK} it was generalized in some aspects, different from the present paper.},
which is identical with the algebra $\op{Sym}(\Delta)$.

The latter has the following basis (we write only the point part of the transformation,
from which it can be uniquely recovered by the prolongation):
 \begin{gather*}
X=\p_x,\quad
W_i^j=\tfrac1{i!}x^i\,\p_{w^j}\ (0\le i+j<k),\\
L=\tfrac1{k!}x^k\,\p_{w^0}+w^0_1\,\p_{w^1}+2w^1_1\,\p_{w^2}+\dots+(k-1)w^{k-2}_1\,\p_{w^{k-1}},\\
R=\tfrac1{(k+1)!}x^{k+1}\,\p_{w^0}+(xw^0_1-(k-1)w^0)\,\p_{w^1}+2(xw^1_1-(k-2)w^1)\,\p_{w^2}+\\
  \quad\quad+3(xw^2_1-(k-3)w^2)\,\p_{w^3}+\dots +(k-1)(xw^{k-2}_1-w^{k-2})\,\p_{w^{k-1}},\\
S_1=x\,\p_x+kw^0\,\p_{w^0}+(k-1)w^1\,\p_{w^1}+\dots+2w^{k-2}\,\p_{w^{k-2}}+w^{k-1}\,\p_{w^{k-1}},\\
S_2=w^0\,\p_{w^0}+2w^1\,\p_{w^1}+3w^2\,\p_{w^2}+\dots+kw^{k-1}\,\p_{w^{k-1}},\\
T=\l\,\p_x+(\l w^0_1-w^1)\,\p_{w^0}+(\l w^1_1-w^2)\,\p_{w^1}+\quad\qquad\\
  \qquad\qquad+\dots+(\l w^{k-2}_1-w^{k-1})\,\p_{w^{k-2}}
  +\tfrac1{k(k+1)}\l^{k+1}\,\p_{w^{k-1}}.
 \end{gather*}
Notice that $\langle  X,W_i^j,L\rangle$ is a translational part and $\langle R,S_1,S_2,T\rangle$
is the stabilizer of the origo (all coordinates vanish) $o\in M$ (to check this for $L$
one has to prolong the formulae to see the term $+\partial_\l$; recall that $\l=w^0_k$).
The non-trivial commutators are the following
 \begin{gather*}
[X,L]=W_{k-1}^0,\ [X,W_i^j]=W_{i-1}^j,\ [L,W_i^j]=-(j+1)W_{i-1}^{j+1}\ (i>0),\\
[X,R]=L,\ [W_i^j,R]=(j+1)(i+j+1-k)W_i^{j+1},\\
[W_i^j,S_1]=(k-i-j)W_i^j,\ [W_i^j,S_2]=(j+1)W_i^j,\ [L,S_2]=L,\\
[L,T]=X,\ [W_i^j,T]=-W_i^{j-1}\ (j>0),\ [R,T]=S_1-S_2,\\
[R,S_1]=-R,\ [R,S_2]=R,\ [T,S_1]=T,\ [T,S_2]=-T.
 \end{gather*}
This algebra is graded with the following weights: $w(X)=w(L)=-1$, $w(W_i^j)=i-k-1$,
$w(R)=w(T)=w(S_1)=w(S_2)=0$.

Even more, it is bi-graded: $b(X)=(-1,0)$, $b(L)=(0,-1)$, $b(W_i^j)=(i+j-k,-j-1)$,
$b(R)=(1,-1)$, $b(T)=(-1,1)$, $b(S_1)=(0,0)$, $b(S_2)=(0,0)$; the grading $w$ is the total
weight of $b$.

It is rather straightforward to check that the above fields are symmetries of the distribution
$\Delta$. This gives the lower bound for $\op{Sym}(\Delta)$. To get the upper bound
one should calculate the \emph{Tanaka algebra\/} of $\Delta$.

Recall \cite{T} that with every distribution one associates the sheaf of graded nilpotent Lie algebras
(called \emph{symbol\/} or \emph{Carnot algebras\/}) $\m=\oplus_{i<0}\g_i$, $\g_i=\Delta_{-i}/\Delta_{-i-1}$,
where $\{\Delta_j\}$ is the weak derived flag of $\Delta$. The point-wise bracket of $\m$
is induced by the commutator of vector fields, so that for every point $x\in M$ we get the Lie algebra
$\m=\m_x$.

The Tanaka prolongation $\hat\m=\oplus\g_i$ is defined as the maximal graded Lie
algebra with the given negative part $\m$ \cite{T}. It can also be defined via graded Lie algebra
cohomology as $\g_i=H^1_i(\m,\m\oplus\g_0\dots\oplus\g_{i-1})$ for $i\ge0$.
In other words, $\g_i$ are constructed successively as maximal subspaces such that all possible
Jacobi identities hold, see \cite{Y$_2$}.

To calculate $\m$ one takes the generators $\D_x^k,\p_\lambda$ of $\Delta$, and
computes all possible brackets. The resulting Carnot algebra is a graded nilpotent
Lie algebra $\mathfrak{n}_{k+1}$ isomorphic to $\langle X,L,W_i^j\rangle$.

It can be described as a truncated double-graded free Lie algebra with fundamental part
of grading $-1$ and rank $2$. In the appendix we demonstrate that the Tanaka prolongation of
$\mathfrak{n}_{k+1}$ is trivial in positive grading. Hence the Tanaka algebra is
$\mathfrak{n}_{k+1}\oplus\g_0$, where $\g_0=\op{gl}(\g_{-1})$.

By the results of \cite{T,K$_1$} and Theorem \ref{THM} this gives the upper bound for
the symmetries of both equations: $\E_k$ (PDE system) and $\mathcal{Y}_k$ (ODE system).
Since it coincides with the algebra of symmetries we already constructed,
our description of $\op{Sym}(\E_k)\simeq\op{Sym}(\Delta)$ is complete.
Theorem \ref{gCE} is proved.

 \begin{rk}
We have demonstrated that $\Delta$ is the most symmetric distribution with the
symbol $\m$ equal to the truncated double-graded free Lie algebra $\mathfrak{n}_{k+1}$.
By Theorem 4 of \cite{AK} this implies that $\Delta$ is Tanaka-flat.
 \end{rk}

We can describe very symmetric systems, for which the corresponding distributions are non-flat.
Let us start with the PDE models $\mathcal{F}_k\subset J^k(\R^2)$:
 $$
\mathcal{F}_k=\Bigl\{u_{k-i,i}=\frac{\lambda^{im+1}}{im+1}:0\le i\le k\Bigr\},
 $$
which coincides with $\E_k$ for $m=1$. For $k=2$ this is (equivalent to) the family
of involutive 2nd order PDE systems
 $$
\Bigl\{u_{xx}=\l,\ u_{xy}=\frac{\l^{m+1}}{m+1},\ u_{yy}=\frac{\l^{2m+1}}{2m+1}\Bigr\}
 $$
with 7-dimensional contact symmetry algebra described
by Cartan in \cite{C$_2$} (the next large after 14-dimensional $G_2$), see also \cite{K$_4$}.

By the calculation similar to the above we get the following
 \begin{theorem}
The algebra of contact symmetries of the equation $\mathcal{F}_k$ for $k>2$ and generic $m$
has dimension $k(k+1)/2+4$. Its basis consists of contact vector fields $X_f$ with
generating functions $f$ as follows
 $$
x^iy^j\ (0\le i+j<k),\ u_x,\ u_y,\ (km+1)\,u-m\,x\,u_x,\ u+m\,y\,u_y.
 $$
 \end{theorem}
By Theorem \ref{THM} we can also represent this algebra as the internal symmetry of the
following underdetermined ODE system obtained from $\mathcal{F}_k$ via reduction by the
Cauchy characteristic:
 $$
w_k^0=\lambda,\ w^1_{k-1}=\tfrac1{m+1}\lambda^{m+1},\ \dots,\ w_1^{k-1}=\tfrac1{(k-1)m+1}\lambda^{(k-1)m+1}.
 $$
(here $w^j=w^j(x)$ are the dependent variables and $w^j_i=\p_x^i(w^j)$).

The symmetry algebra has a filtration with the corresponding graded Lie algebra being
the semi-direct product $\mathfrak{n}_{k+1}\rtimes\R^2$, where
$\R^2$ is the diagonal part of $\g_0$. This associated graded  algebra
is a subalgebra of the Tanaka algebra $\hat{\mathfrak{n}}_{k+1}$ of the flat model.

It is possible to show by the methods of \cite{K$_4$} that the corresponding distribution
$\Delta$ is sub-maximal symmetric with the Tanaka algebra $\hat{\mathfrak{n}}_{k+1}$.

\section{Application II}\label{S7}

The PDE system $\E_k$ has solution space $\op{Sol}(\E_k)$ that is parametrized by 1 function
of 1 argument and $\dim M-2=\frac{k(k+1)}2$ constants (these are the so-called Lie class $\omega=1$
systems \cite{L,K$_2$,K$_3$}).

More general systems can be treated with the proposed technique too.
Consider, for instance, the following system $\mathcal{R}_k^m\subset J^k(\R^2)$, $m<k$, for which the
right-hand side is the $m$-th tangent cone (of $\op{dim}=m+1$ in $\R^{k+1}$)\footnote{This
$\R^{k+1}(u_{k0},\dots,u_{0k})=S^k(\R^2)^*$ is the fiber of $\pi_{k,k-1}:J^k(\R^2)\to J^{k-1}(\R^2)$.}
of the normal projective curve defining $\E_k=\mathcal{R}_k^0$:
 $$
\mathcal{R}_k^m=\Bigl\{u_{k-i,i}=\frac{\l^{i+1}}{i+1}+\sum_{j=1}^{\op{min}(m,i+1)}
\frac{i!}{(i+1-j)!}\l^{i+1-j}\zeta_j : 0\le i\le k\Bigr\}
 $$
(the PDE system is obtained by excluding the additional parameters $\l,\zeta_1,\dots,\zeta_m$
and obtaining relations on the jet-variables $u_{k-i,i}$).

For instance for $k=3$, $m=2$ the system $\mathcal{R}_3^2$ looks so
(this is one highly nonlinear PDE of the 3rd order)
 \begin{multline*}
u_{xxx}=\l+\zeta_1,\ u_{xxy}=\frac{\l^2}2+\l\zeta_1+\zeta_2,\\
u_{xyy}=\frac{\l^3}3+\l^2\zeta_1+2\l\zeta_2,\ u_{yyy}=\frac{\l^4}4+\l^3\zeta_1+3\l^2\zeta_2.
 \end{multline*}
The system $\mathcal{R}_2^1$ is the famous Goursat parabolic PDE on the plane; its
contact symmetry group is $G_2$ the same as for (\ref{Ceq}).

The system $\mathcal{R}_k^m$ is involutive and its characteristic variety (both complex and real)
consists of 1 point with multiplicity $m+1$.
The solution space $\op{Sol}(\mathcal{R}_k^m)$
is parametrized by $\oo=m+1$ functions of 1 argument (and some constants).
There is a reduction of this parabolic system via a characteristic involutive distribution
(no longer a space of Cauchy characteristics) to a distribution $\tilde\Delta$ on the quotient
that de-prolongs to a rank 2 distribution $\Delta$ (Monge system);
this generalizes \cite{G}.

This reduction makes a bijection between contact symmetries of
the PDE system and the ordinary symmetries of $\Delta$.
While this idea is more general, we will indicate only how it applies to our system
$\mathcal{R}_k^m$.

 \begin{theorem}\label{GcE}
The algebra of contact symmetries of the equation $\mathcal{R}_k^m$ for $k>2$, $0<m<k$,
has dimension $k(k+1)/2+6$ and is isomorphic to the same Lie algebra
$\mathfrak{n}_{k+1}\rtimes\op{gl}_2$ as in Theorem \ref{gCE}.
 \end{theorem}

 \begin{proof}
Let us first show how to associate a rank 2 distribution to such a parabolic system of PDEs.

Consider the Cartan distribution $\CC$ of $\mathcal{R}_k^m$. It is generated by two truncated
total derivatives $\D_x^k,\D_y^k$ and the vertical fields $\p_\l,\p_{\zeta_1},\dots,\p_{\zeta_m}$.
The sub-distribution $\Pi^m=\langle\p_{\zeta_1},\dots,\p_{\zeta_m}\rangle$ is integrable
and we would like to quotient by it.

It however does not commute with $\CC$, so we need to add the commutators
 $$
\op{ad}_\Pi^\infty(\CC)=\CC+\langle\eta_1,\dots,\eta_m\rangle,
 $$
where $\eta_1=\sum_{i=0}^k\l^i\p_{u_{k-i,i}}$, $\eta_{1+p}=\p_\l^p(\eta_1)$.

The quotient manifold is $\tilde M=\R^{k(k+1)/2+3}(x,y,\l,u_\z:|\z|<k)$ and it is equipped with the
distribution $\tilde\Delta=\op{ad}_\Pi^\infty(\CC)/\Pi$ generated by the truncated total
derivatives evaluated at $\zeta_1=\dots=\zeta_m=0$, denoted $\bar\D_x^k$, $\bar\D_y^k$, and
by the fields $\eta_1,\dots,\eta_m$, $\p_\l$.
For instance, for $k=3,m=2$
 \begin{gather*}
\bar\D_x=\p_x+p\,\p_u+r\,\p_p+s\,\p_q+\l\,\p_r+\tfrac12\l^2\p_s+\tfrac13\l^3\p_t,\ \\
\bar\D_y=\p_y+q\,\p_u+s\,\p_p+t\,\p_q+\tfrac12\l^2\p_r+\tfrac13\l^3\p_s+\tfrac14\l^4\p_t,\\
\eta_1=\p_r+\l\,\p_s+\l^2\p_t,\qquad \eta_2=\p_s+2\l\,\p_t
 \end{gather*}

The vector $\xi=\bar\D_y-\l\bar\D_x$ is the Cauchy
characteristic of the distribution $\tilde\Delta$.
Let $(M,\Delta_+)=(\tilde M,\tilde\Delta)/\xi$ be the
local quotient. It is easy to see that $\Delta_+$ is the $m$-th derived
(both weak and strong) distribution of the rank 2 distribution
$\Delta=\langle\bar\D_x,\p_\l\rangle$.

Thus any contact symmetry of $\CC$ descends to an ordinary symmetry of $\Delta$.
Conversely, $\op{Sym}(\Delta)=\op{Sym}(\Delta_+)$ because prolongation
preserves the symmetries, and then by Theorem \ref{THM} the latter coincide with the contact
symmetries of $\tilde\Delta$. This distribution corresponds to the equation $\E_k=\mathcal{R}_k^0$,
from which by taking the $m$-th tangent cone we obtain our system $\mathcal{R}_k^m$.
Thus any ordinary symmetry of $\Delta$ uniquely lifts to a contact symmetry of $\CC$.

The claim now follows from Theorem \ref{gCE} on the symmetries of $\E_k$.
 \end{proof}

\section{The case of $n>2$ independent variables}\label{S8}

It is easy to see that the generalized Lie-B\"acklund theorem, as it was stated in Introduction,
fails if we allow equations of the first order without non-degeneracy assumptions.
Indeed, this is so with any class $\oo=1$ system of the type $2E_2+E_1$ in $n>2$ independent variables.

For instance, we can take (\ref{Ceq}) as $2E_2$ and trivial $E_1$:
 $$
u_{xx}=\l,\ u_{xy}=\frac{\l^2}2,\ u_{yy}=\frac{\l^3}3,\ u_z=0.
 $$
Then $z\mapsto Z(x,y,z,u,u_x,u_y,u_z)$ is a contact symmetry, inducing the trivial
transformation of the reduction $(M^5,\Delta_\text{HC})$, which is the Hilbert-Cartan equation.
In fact, the general contact symmetry $X_f$ has generating function $f=f_0+u_z\cdot \tilde f$,
where $f_0=f_0(x,y,u,u_x,u_y)$ is the 14-parametric generating function corresponding to
$G_2$-action on (\ref{Ceq}), and $\tilde f$ is an arbitrary function on $J^1(\R^3)$.
Thus the generalized Lie-B\"acklund theorem fails: the map (\ref{LBT}) is not injective.

Let's consider a similar system of the second order:
 $$
u_{xx}=\l,\ u_{xy}=\frac{\l^2}2,\ u_{yy}=\frac{\l^3}3,\ u_{xz}=0,\ u_{yz}=0,\ u_{zz}=0.
 $$
Its reduction (along rank 2 distribution $\Pi$)
is the following rank 2 distribution in 6D:
$(\bar M,\bar\Delta)=(M^5,\Delta_\text{HC})\times(\R,0)$.
Let $t$ be the first integral of the distribution $\bar\Delta$ (it is not completely non-holonomic).
The symmetries are $t\mapsto T(t)$ and $G_2$-transformations of the first factor, with
coefficients parametrized as functions of $t$. Thus in total
$\op{Sym}(\bar M,\bar\Delta)$ is parametrized by 15 functions of 1 variable.

A direct calculation shows that the contact algebra is precisely the same:
$(\op{Lie}(G_2)\oplus\R)^\R$, so the generalized Lie-B\"acklund theorem holds.

Similarly this theorem holds for higher order equations. For instance,
the system of type $3E_3+3E_2$
 $$
u_{xxx}=\l,\ u_{xxy}=\frac{\l^2}2,\ u_{xyy}=\frac{\l^3}3,\ u_{yyy}=\frac{\l^4}4,\
u_{xz}=0,\ u_{yz}=0,\ u_{zz}=0
 $$
reduces to $(\bar M^9,\bar\Delta^2)=(M_1,\Delta_1)\times(\R,0)$, and both
contact symmetry of the PDE system and the internal symmetry of the reduction are
parametrized by 13 functions of 1 variable. More precisely, these algebras are both
equal to $(\g\oplus\R)^\R$, where $\g$ is the algebra of symmetries of $(M_1,\Delta_1)$
equivalent to the Monge system of 2 equations on 3 unknowns (the indices denote the number
of $x$-derivatives) studied in Section \ref{S6}:
 $$
v_3=\lambda,\ w_2=\tfrac12\lambda^2,\ z_1=\tfrac13\lambda^3.
 $$

On the other hand if we consider the system of the 3rd order $9E_3$
 \begin{multline*}
u_{xxx}=\l,\ u_{xxy}=\frac{\l^2}2,\ u_{xyy}=\frac{\l^3}3,\ u_{yyy}=\frac{\l^4}4,\\
u_{xxz}=0,\ u_{xyz}=0,\ u_{yyz}=0,\ u_{xzz}=0,\ u_{yzz}=0,\ u_{zzz}=0,
 \end{multline*}
its contact symmetry algebra is 21-dimensional, while the internal symmetry algebra of the
reduction (the distribution there is not completely holonomic) is infinite-dimensional.
Thus the generalized Lie-B\"acklund theorem fails: the map (\ref{LBT})
is not surjective.

All these systems have a kind of degeneracy, so let us consider a totally non-linear
(and non-degenerate) system, for which our Theorems guarantee that the generalized
Lie-B\"acklund theorem holds.

We start with the equations of pure order 2 and Lie class $\oo=1$, for which
the result, that contact external symmetries of the PDE system
and the internal symmetries of the reduction coincide, follows also from \cite{Y$_1$}.
Consider the following example.
 $$
\E=\Bigl\{u_{ij}=\frac{\l^{m_i+m_j+1}}{m_i+m_j+1}:1\le i\le j\le n\Bigr\}.
 $$
This system is involutive and, if $m_i\ne m_j$ for $i\ne j$, it has no first integrals.
We can always achieve $m_1=0$ by re-parametrization. The reduction of $\E$ is given by
 \begin{equation}\label{kl}
v^2_x=\frac{(v^1_{xx})^{m_2+1}}{m_2+1},\ v^3_x=\frac{(v^1_{xx})^{m_3+1}}{m_3+1},\ \dots\
v^n_x=\frac{(v^1_{xx})^{m_n+1}}{m_n+1}
 \end{equation}
(the superscript numbers the unknown functions). The internal symmetry algebra is maximal
for $m_2=1$, $m_3=2$, \dots, $m_n=n-1$, and has dimension $(2n+5)$;
the proof of this fact is similar to the proof of Theorem \ref{gCE} and so is omitted.
In fact, the rank 2 distribution of ODE system (\ref{kl}) with the prescribed parameters
is internally equivalent to the most symmetric Monge equation \cite{AK}
(subscripts denote the derivatives)
 $$
y_1=(z_n)^2
 $$
(its symmetry algebra with notation $\mathfrak{t}_{1,n}=\hat{\mathfrak{p}}_{n+3}$
was studied in \cite{AK}).

The symmetry algebra is the same for the ODE and PDE systems for all $m_i$, but
for generic parameters the algebra has smaller dimension.

Now the above PDE system $\E$ can be generalized to higher orders. For the third order
it writes as
 $$
\Bigl\{u_{ijk}=\frac{\l^{m_i+m_j+m_k+1}}{m_i+m_j+m_k+1}:1\le i\le j\le k\le n\Bigr\}.
 $$
It is involutive, but to achieve complete non-holonomy (no first integrals),
we have to assume that all the numbers $m_j+m_k$, $1\le j\le k\le n$, are different.
The reduction can be again easily described.

Consider a particular case $n=3$, $m_1=0$, $m_2=1$, $m_3=4$, satisfying the above restriction.
The symmetry algebra of both ODE reduction and the PDE system is 15-dimensional.

Similarly for $n=4$ independent variables, and parameters $m_1=0$, $m_2=1$, $m_3=4$,
$m_4=15$, we calculate both symmetry algebras to have dimension 21.

\smallskip

We conclude that the generalized Lie-B\"acklund theorem holds in many interesting cases.
The calculations are easier for the reduced ODE models, where Tanaka theory helps to
calculate symmetries by using the algebraic methods.

\appendix
\section{Truncated double-graded free Lie algebra}\label{A}

A double-graded free Lie algebra is such a graded Lie algebra
$\mathfrak{n}_\infty=\oplus_{i<0}\g_i$ that its fundamental space $\g_{-1}$
is a direct sum of two subspaces
$\Pi_1,\Pi_2$, and its commutators freely generate the whole algebra with no other relations than
the commutation of $\op{ad}_{v_1}$, $\op{ad}_{v_2}$ for $v_i\in\Pi_i$
(whence $\Z$-grading can be refined to $\Z\oplus\Z$-grading).

Let us specify this only in the case of current interest $\dim\g_{-1}=2$, when the fundamental space
has basis $e_{10},e_{01}$ according to the above splitting.
Then $\g_{-2}$ is generated by $e_{11}=[e_{10},e_{01}]$,
$\g_{-3}$ by $e_{21}=[e_{10},e_{11}]$, $e_{12}=[e_{01},e_{11}]$ etc.
The only relations this infinite-dimensional Lie algebra $\mathfrak{n}_\infty$ admits are
$[e_{10},e_{i,j+1}]=[e_{01},e_{i+1,j}]=e_{i+1,j+1}$.

Thus $\g_{-k-1}=\langle e_{k,1},e_{k-1,2},\dots,e_{1,k}\rangle$ has dimension $k$.
It's easy to check that $\mathfrak{n}_\infty$ is a Lie algebra.
Many interesting graded nilpotent Lie algebras are quotients of this algebra
(see some in \cite{AK}).

The truncated algebra is $\mathfrak{n}_k=\g_{-k}\oplus\dots\oplus\g_{-1}$ (the brackets between
$\g_{-i}$ and $\g_{-j}$ are zero if $i+j>k$). Our goal is to
calculate the Tanaka prolongation of $\mathfrak{n}_k$
(beware: this sub-script $k$ is not a grading).

 \begin{theorem}
For $k>3$ the Tanaka prolongation of $\mathfrak{n}_k$ is supported in non-positive grading:
$\hat{\mathfrak{n}}_k=\mathfrak{n}_k\oplus\g_0$, $\g_0=\op{gl}(\g_{-1})$.
 \end{theorem}

We consider prolongation in the usual sense, not preserving double-grading (in particular
$\g_0$ consists of grading zero derivations, not necessary preserving the splitting of $\g_{-1}$).

 \begin{proof}
We first state that $\g_0=\op{gl}(\g_{-1})$ is the maximal possible algebra
of grading preserving derivations.
The easiest way to see this is to calculate the prolongation of an element
 $$
h=\begin{bmatrix}a & b \\ c & d \end{bmatrix}\in\g_0
 $$
(the matrix in the basis $e_{10},e_{01}$) and check that as derivation it preserves the
relations $[\op{ad}_{e_{10}},\op{ad}_{e_{01}}]=0$.

This readily follows from the following formula that can be proved by induction ($1\le i\le k-1$):
 $$
h(e_{k-i,i})=(i-1)b\,e_{k-i+1,i-1}+((k-i)a+i\,d)\,e_{k-i,i}+(k-i-1)c\,e_{k-i-1,i+1}.
 $$

Let now $\oo\in\g_1$. Since $\g_{-1}$ is fundamental (generates $\mathfrak{n}_k$), this
element is uniquely determined by specifying $\oo(e_{10})=h'$, $\oo(e_{01})=h''$,
where we denote $\oo(\xi)=[\oo,\xi]$ and the elements of $\g_0$ have the form
 $$
h'=\begin{bmatrix}a' & b' \\ c' & d' \end{bmatrix},\quad
h''=\begin{bmatrix}a'' & b'' \\ c'' & d'' \end{bmatrix}.
 $$
By the Leibniz rule we calculate:
 \begin{gather*}
\oo(e_{11})=(b'-a'')\,e_{10}+(d'-c'')\,e_{01},\\
\oo(e_{21})=(a'+2d'-c'')\,e_{11},\quad \oo(e_{12})=(2a''+d''-b')\,e_{11},\\
\oo(e_{31})=(3a'+3d'-c'')\,e_{21}+c' e_{12},\\
\oo(e_{22})=(2a''+d'')\,e_{21}+(a'+2d')\,e_{12},\\
\oo(e_{13})=b'' e_{21}+(3a''+3d''-b')\,e_{12}.
 \end{gather*}
In calculation of $\oo(e_{22})$ we can use two representations $[e_{10},e_{12}]=[e_{01},e_{21}]$
and this gives the same result. Thus if we truncate on the level $k=3$ (so that $e_{31}=e_{22}=e_{13}=0$),
then all coefficients to the right in the last three lines vanish and we obtain that $\g_1$
is 2-dimensional (in agreement with the known grading of the exceptional Lie group $G_2$).

Let now $k>3$ and we calculate the action on $\g_{-5}$. We get
 \begin{gather*}
\oo(e_{41})=\oo([e_{10},e_{31}])=(6a'+4d'-c'')\,e_{31}+3c' e_{22},\\
\oo(e_{14})=\oo([e_{01},e_{13}])=3b'' e_{22}+(4a''+6d''-b')\,e_{13}
 \end{gather*}
uniquely and
 \begin{align*}
\oo(e_{32})&=\oo([e_{10},e_{22}])=(2a''+d''+b')\,e_{31}+(3a'+4d')\,e_{22}+c' e_{13}\\
&=\oo([e_{01},e_{31}])=(3a''+d'')\,e_{31}+(3a'+3d'+c'')\,e_{22}+c' e_{13},
 \end{align*}
 \begin{align*}
\oo(e_{23})&=\oo([e_{10},e_{13}])=b'' e_{31}+(3a''+3d''+b')\,e_{22}+(a'+3d')\,e_{13} \\
&=\oo([e_{01},e_{22}])=b'' e_{31}+(4a''+3d'')\,e_{22}+(a'+2d'+c'')\,e_{13}
 \end{align*}
non-uniquely, which implies $b'=a''$, $c''=d'$.

Further calculations show that the derivation respects the higher commutation relations,
and we obtain by induction:
 \begin{multline*}
\oo(e_{k-i,i+1})=\frac{i(i-1)}2\,b''\,e_{k-i+1,i-1}
 +\Bigl((k-i)i\,a''+\frac{i(i+1)}2\,d''\Bigr)\,e_{k-i,i}\\
+\Bigl(\frac{(k-i)(k-i-1)}2\,a'+(i+1)(k-i-1)d'\Bigr)\,e_{k-i-1,i+1}\\
 +\frac{(k-i-1)(k-i-2)}2\,c'\,e_{k-i-2,i+2}.
 \end{multline*}
Now it's obvious that truncation on the level $k$, i.e. letting $e_{k+1-i,i}=0$
for $1\le i\le k$,
forces $h'=h''=0$, so that $\g_1=0$.
 \end{proof}


\end{document}